\documentclass[12pt, a4paper]{article}
\usepackage[margin=1in, top=1.5in, bottom=1.5in]{geometry}
\usepackage{subcaption}
\usepackage{graphicx}
\usepackage{float}
\usepackage{titlesec}
\usepackage{amsmath}
\usepackage{amssymb}
\usepackage{amsthm}
\usepackage{cases}
\usepackage[english]{babel}
\usepackage{hyperref}
\usepackage{enumitem}
\usepackage{authblk}
\usepackage{dsfont}
\usepackage{sectsty}
\usepackage{xcolor}

\usepackage{todonotes}

\theoremstyle{break}
	\newtheorem{thm}{Theorem}[section]
\theoremstyle{break}
    \newtheorem{prp}{Proposition}[section]
\theoremstyle{break}
	\newtheorem{asm}{Assumption}[section]
\theoremstyle{break}
	\newtheorem{rem}{Remark}[section]
\theoremstyle{break}
	\newtheorem{lem}{Lemma}[section]
\theoremstyle{break}
	
    \theoremstyle{break}
	
\theoremstyle{break}
	
\theoremstyle{break}
	\newtheorem{example}{Example}[section]

\makeatletter
\newcommand{\mylabel}[2]{#2\def\@currentlabel{#2}\label{#1}}
\makeatother

\counterwithin{equation}{section}

\title{Eigenvalues and eigenfunctions of the non-local dispersal with Neumann-type boundary condition and symmetric kernel}

\author[1]{Maciej Tadej \thanks{maciej.tadej@math.uni.wroc.pl}}

\affil[1]{\textit{Mathematical Institute, University of Wrocław, pl. Grunwaldzki 2, 50-384 Wrocław, Poland}}

\date{}

\begin{document}

\maketitle

\hspace{20pt}

\begin{abstract}
    In this paper, we investigate the eigenvalue problem for a non-local dispersal operator defined on a bounded spatial domain with Neumann-type boundary conditions. Unlike the classical Laplacian, the non-local operator lacks compactness, which gives rise to an essential (continuous) spectrum and severely complicates the study of its discrete eigenvalues. The main contribution of this work is the rigorous variational construction of a finite or infinite sequence of non-trivial eigenfunctions corresponding to isolated eigenvalues located strictly above the continuous spectrum. Since the existence of these eigenvalues is not generally guaranteed due to the potential collapse of the spectral gap, we establish explicit sufficient geometric conditions linking the domain's size and geometry with the variance of the dispersal kernel that assure the emergence of the principal eigenvalue. Subsequent eigenpairs are constructed inductively via finite-dimensional Galerkin approximations, utilizing a decomposition of the operator into compact and multiplicative part, we prove the strong convergence of the minimizing sequences.
\end{abstract}


\newpage

\section{Introduction}

In this work, we study the non-local dispersal operator $\mathcal{L}$, which serves as a non-local analogue of the classical Laplacian with Neumann boundary conditions. It is defined by the formula
\begin{equation}
    \label{EQ:NonLocalNeumann}
    \mathcal{L}v(x) = \int_{\Omega} J(x-y) (v(y) - v(x)) \, dy \quad x \in \Omega.
\end{equation}
The main objective of this work is to construct pairs $(\beta, v) \in \mathbb{R}\times L^2(\Omega)$ solving the non-local eigenvalue problem given below
\begin{equation}
    \label{eq:nonlocal_eigenvalue_problem}
    \mathcal{L} v = \beta v.
\end{equation}
We aim to achieve this goal by means of the calculus of variations. In particular we seek the solutions of the following optimization problem

\begin{equation}
    \beta = \sup_{v \in L^2(\Omega)} \langle \mathcal{L} v, v\rangle.
\end{equation}

We assume the following properties of the integral kernel $J$.

\begin{asm} 
    \label{asm:KernelProperties}
    The integral kernel $J \in C^0(\mathbb{R}^n)$ satisfies
    \begin{enumerate}[label=(\textit{J\arabic*})]
        \item \textbf{Positivity:} $J(z) \geq 0$ for all $z \in \mathbb{R}^n$, and $J(0) > 0$.
        \item \textbf{Symmetry:} $J(z) = J(-z)$ for all $z \in \mathbb{R}^n$.
        \item \textbf{Normalization:} $\int_{\mathbb{R}^n} J(z) \, dz = 1$.
        \item \textbf{Finite Second Moment:} $\int_{\mathbb{R}^n} J(z)|z|^2 \, dz < \infty$.
    \end{enumerate}
\end{asm}

Furthermore, we impose one of the supplementary assumptions on the domain $\Omega$. The first one is synthetic, but it allows us to demonstrate Example \ref{ex:generalized_exponential}.

\begin{asm}
    \label{asm:DomainProperties_1}
    The domain $\Omega \subset \mathbb{R}^n$ is open and bounded. There exist $\Omega_1, \Omega_2 \subset \Omega$ such that $\Omega_1 \cap \Omega_2 = \emptyset$, $\Omega_1 \cup \Omega_2 = \Omega$ and the following inequality holds true
    \begin{equation}\label{eq:sufficient_cond}
        \frac{4}{|\Omega|} \iint_{\Omega_1 \times \Omega_2} J(x-y) \, dx dy < \min_{x \in \overline{\Omega}} \int_\Omega J(x-y) \, dy.
    \end{equation}
\end{asm}
Conclusions from Example \ref{ex:generalized_exponential} brought us to the more useful condition, needed to prove existence of the first non-trivial eigenfunction.

\begin{asm}
    \label{asm:DomainProperties_2}
    The domain $\Omega \subset \mathbb{R}^n$ is open and bounded. We assume that $\Omega$ is large enough so the following inequality holds true
    \begin{equation} \label{eq:variance_condition}
        \frac{1}{|\Omega|^{2/n}}\frac{m_2(J)}{2 \, I_{max}} < \min_{x \in \overline{\Omega}} \int_\Omega J(x-y) \, dy,
    \end{equation}
    where we denoted the maximal moment of inertia by $I_{max} = \max_{1\leq i \leq n} \int_{\Omega_0} x_i^2 \:dx$ and a reference domain $\Omega_0 = \frac{1}{|\Omega|^{1/n}}\Omega$.
\end{asm}
To produce further eigenpairs, we no longer have a nice expression consisting of maximal moment of inertia. Instead of that, we state the criterion with an implicit, geometric constant.

\begin{asm}
    \label{asm:DomainProperties_3}
    The domain $\Omega \subset \mathbb{R}^n$ is open and bounded. Let $k \geq 1$ be a given integer. We assume that $\Omega$ is large enough so the following inequality holds true
    \begin{equation} \label{eq:variance_condition}
        C(k,n,\Omega_0)\frac{m_2(J)}{|\Omega|^{2/n}} < \min_{x \in \overline{\Omega}} \int_\Omega J(x-y) \, dy,
    \end{equation}
    where we denoted the reference domain $\Omega_0 = \frac{1}{|\Omega|^{1/n}}\Omega$ and a fixed geometric constant $C(k, n, \Omega_0) > 0$.
\end{asm}

\paragraph{Literature review}

A central challenge in the study of non-local dispersal is the lack of compactness, which fundamentally changes the spectral structure. The implications of this property for nonautonomous, periodic, and random environments were extensively investigated by Shen, Vickers, and Hutson \cite{ShenVickers2007, HutsonShenVickers2008}. Their work established specific conditions under which a principal eigenvalue is guaranteed to exist, and demonstrated that temporal fluctuations strictly increase the principal Lyapunov exponent, yielding the ecological conclusion that heterogeneous environments mathematically facilitate species invasion.

Foundational efforts to understand the eigenvalue problem for such operators were made by Coville, Rossi, and their collaborators (see e.g., \cite{ChasseigneChavesRossi2006, Coville2010, GarciaMelianRossi2009}), as well as comprehensively detailed in the monograph by Andreu-Vaillo et al. \cite{AndreuVaillo2010}. This line of research primarily focused on establishing maximum principles, exploring the asymptotic behavior of solutions, and analyzing the criteria for the existence of principal eigenvalues, often highlighting the fundamental differences between local and non-local operators under various boundary conditions. We highlight the following Theorems that are crucial to this work
\begin{thm}[Principal eigenvalue of the Neumann type dispersal]
    \label{thm:beta1_bounds}
    Let $\beta_1$ be the first non-trivial eigenvalue of $\mathcal{L}_N$. This value is well defined and satisfies the following bounds
    \begin{equation}
        \max_{x\in\overline{\Omega}} b(x) \leq \beta_1 < 0.
    \end{equation}
\end{thm}
\begin{proof}
    For the detailed proofs of lower and upper bounds of $\beta_1$ we refer the reader to Proposition 3.4 and Lemma 3.5 in \cite{AndreuVaillo2010}.
\end{proof}
Note, that even though Theorem \ref{thm:beta1_bounds} states that the candidate for eigenvalue $\beta_1$ is well defined, it is not generally known whether there exists a corresponding eigenfunction. In fact, it might be that such a function does not exists, see \cite{Coville2010, ShenZhang2010} for examples. Situation in the hostile environment case (Dirichlet) is a bit more optimistic.
\begin{thm}[Principal eigenvalue of the Dirichlet type dispersal]
    \label{thm:dirichlet_principal_eigenvalue}
    Let $\mathcal{L}_D$ be the non-local Dirichlet operator. There exists a principal eigenvalue $\beta_1$ associated with $\mathcal{L}_D$, and a corresponding non-negative eigenfunction $v_0 \ge 0$. Moreover, $v_0$ is strictly positive in $\Omega$ (with a positive continuous extension to $\overline{\Omega}$) and vanishes strictly outside $\Omega$, meaning the boundary value is not taken in the usual classical sense and a discontinuity occurs on $\partial\Omega$.
\end{thm}
\begin{proof}
    For the detailed proof of the existence of the principal eigenvalue and the exact spatial properties of the corresponding non-negative eigenfunction $v_0$, we refer the reader to Section 2.1.1, Proposition 2.3, and Remark 2.4 in \cite{AndreuVaillo2010}.
\end{proof}
In more recent work, Foss, Radu, and Wright \cite{FossRaduWright2018} established a robust framework for minimizing non-local energy functionals. Their work demonstrates how to derive non-local Euler-Lagrange equations and prove the existence and regularity of minimizers using joint convexity and non-local Poincar\'e inequalities, notably without relying on the restrictive growth conditions required in classical Sobolev spaces.

Subsequent research focused on generalized, variational characterizations of the principal eigenvalue. Berestycki, Coville, and Vo \cite{BerestyckiCovilleVo2016} established the equivalence of various spectral definitions and analyzed kernel scaling asymptotics. Li, Coville, and Wang \cite{LiCovilleWang2017} further utilized the strong maximum principle to derive exact max-min characterizations of the principal eigenvalue, rigorously proving its algebraic simplicity and isolation. Recently, Wang and Zhang \cite{WangZhang2026} extended this framework to cooperative non-local systems. By approximating the generalized principal eigenvalue through monotonic control matrices, they demonstrated that this spectral bound strictly dictates the threshold dynamics and global stability of the coupled systems.

Parallel to the development of variational formulations, extensive research has been devoted to the parametric dependence of the principal spectrum point. Shen and Xie \cite{ShenXie2015} systematically analyzed the impact of spatial inhomogeneity, dispersal rate, and dispersal distance across Dirichlet, Neumann, and periodic boundary conditions, establishing rigorous asymptotic limits and persistence criteria. Bao and Shen \cite{BaoShen2016} subsequently extended this framework to time-periodic cooperative systems, deriving explicit geometric conditions, such as the vanishing inhomogeneity condition and small dispersal distance limits, that guarantee the existence and algebraic simplicity of the principal eigenvalue. More recently, Vo \cite{Vo2022} addressed the distinct analytical challenges posed by time-periodic non-local Neumann operators. Because the principal eigenvalue in the non-local Neumann setting lacks domain monotonicity, Vo developed new techniques to determine its precise asymptotic limits with respect to varying dispersal scales and established the exact spectral criteria required for the validity of the associated maximum principle.

\paragraph{Main result}

We have delivered two crucial results for the spectral theory of non-local dispersal operators. The first one established existence of the spectral gap

\begin{thm}[Spectral Gap Theorem]
    \label{thm:spectral_gap_theorem}
    Assume that the kernel $J$ satisfies Assumption \ref{asm:KernelProperties} and that domain $\Omega$ satisfies Assumption \ref{asm:DomainProperties_3} with fixed $k \geq 1$. Then the following inequality holds true
    \begin{equation}
        \sup \sigma_c < \beta_k < \ldots < \beta_0.
    \end{equation}
\end{thm}

\begin{proof}
    The proof relies on the scaling argument established in Lemma \ref{lem:large_domain_gap_k}. By Assumption \ref{asm:DomainProperties_3}, the domain $\Omega$ can be viewed as a scaled version of the reference domain $\Omega_0 = \frac{1}{|\Omega|^{1/n}}\Omega$ with the scaling factor $\lambda = |\Omega|^{1/n}$. 
    
    Following the construction in Lemma \ref{lem:large_domain_gap_k}, we can find a $(k+1)$-dimensional subspace $W \subset C^1(\overline{\Omega})$ and a uniform Lipschitz constant $L_k$ such that for a suitably chosen test function $v \in W \cap A_k$, the energy functional is strictly bounded from below by the variance of the kernel. Specifically, using the established scaling of the squared differences \eqref{eq:lipschitz_scaling}, the energy satisfies:
    \begin{equation*}
        \beta_k = \sup_{v \in A_k} \langle \mathcal{L}v, v \rangle \geq - C(k,n,\Omega_0) \frac{m_2(J)}{|\Omega|^{2/n}},
    \end{equation*}
    where $C(k,n,\Omega_0) = L_k / 2$. 
    
    Applying the strict geometric condition \eqref{eq:variance_condition} from Assumption \ref{asm:DomainProperties_3}, we obtain:
    \begin{equation*}
        - C(k,n,\Omega_0) \frac{m_2(J)}{|\Omega|^{2/n}} > -\min_{x \in \overline{\Omega}} \int_\Omega J(x-y) \, dy.
    \end{equation*}
    By Theorem \ref{thm:contninous_spectrum}, the right-hand side is exactly equal to $\sup \sigma_c$. Therefore, $\beta_k > \sup \sigma_c$, which guarantees the existence of the $k$-th spectral gap. 

    Furthermore, since the sets of admissible functions are strictly nested, i.e., $A_k \subset A_{k-1} \subset \dots \subset A_1$, the suprema naturally satisfy the monotonic ordering $\beta_k \leq \beta_{k-1} \leq \dots \leq \beta_1$. Finally, Theorem \ref{thm:beta1_bounds} ensures that $\beta_1 < 0 = \beta_0$, which concludes the proof.
\end{proof}

Working in the framework of Assumption \ref{asm:KernelProperties} together with Assumption \ref{asm:DomainProperties_1}, \ref{asm:DomainProperties_2} or \ref{asm:DomainProperties_3}, we obtained the following spectral theorem for the non-local dispersal operators

\begin{thm}[Non-local Spectral Theorem]
    \label{thm:spectral_theorem}
    Assume that the kernel $J$ satisfies Assumption \ref{asm:KernelProperties} and that domain $\Omega$ satisfies either Assumption \ref{asm:DomainProperties_1}, \ref{asm:DomainProperties_2} or \ref{asm:DomainProperties_3}. Then there exists a finite or infinite sequence of pairs $(\beta_k, v_k) \in \mathbb{R} \times L^2(\Omega)$ solving equation \eqref{eq:nonlocal_eigenvalue_problem}. Moreover, the sequence of eigenvalues $\{\beta_k\}_{k = 0}^M$ satisfies
    \begin{equation}
        -\min_{x \in \overline{\Omega}} \int_\Omega J(x-y) \, dy < \beta_M < \ldots < \beta_0 = 0,
    \end{equation}
    where $M \geq 1$.
\end{thm}

\begin{proof}[Proof of Theorem \ref{thm:spectral_theorem}]
    We construct the sequence of eigenpairs by induction. For $k=0$, the existence of the trivial eigenpair $(\beta_0, v_0) = (0, |\Omega|^{-1/2})$ follows directly from the mass-conserving property of the non-local operator $\mathcal{L}$ with Neumann-type boundary conditions (Remark 2.1).

    For the principal non-trivial eigenvalue ($k=1$), Proposition \ref{prop:spectral_gap} (under Assumption \ref{asm:DomainProperties_1}) or the scaling arguments (under Assumptions \ref{asm:DomainProperties_2} and \ref{asm:DomainProperties_3}) guarantee the existence of a strictly positive spectral gap: $\sup \sigma_c < \beta_1$. By Lemma \ref{lem:convergence_principal_eigenvalue} and Lemma \ref{lem:convergence_principal_eigenfunction}, the sequence of finite-dimensional Galerkin approximations $(\beta_1^N, v_1^N)$ algebraically defined in Proposition \ref{prp:finite_dimensional_minimizer} converges strongly in $L^2(\Omega)$ to a valid eigenpair $(\beta_1, v_1) \in \mathbb{R} \times A_1$.

    Suppose that for some integer $k \geq 1$, there exist $k$ normalized, mutually orthogonal eigenfunctions $v_1, \dots, v_k$ corresponding to eigenvalues $\beta_k \leq \dots \leq \beta_1 < 0$. If the domain size and the kernel variance satisfy the criteria to maintain the spectral gap for the $(k+1)$-th mode, i.e.,
    \begin{equation*}
        \sup \sigma_c < \sup_{v \in A_{k+1}} \langle \mathcal{L}v, v \rangle =: \beta_{k+1},
    \end{equation*}
    then Lemma \ref{lem:inductive_step} rigorously ensures that the Galerkin approximations on the orthogonally restricted subspace $A_{k+1}$ converge strongly to a new eigenpair $(\beta_{k+1}, v_{k+1})$. 
    
    This inductive process can be continued as long as the spectral gap condition holds. Let $M \geq 1$ denote the number of such eigenvalues. Because the operator $\mathcal{L}$ is bounded, its discrete spectrum located strictly above the essential spectrum $\sigma_c$ can only consist of isolated eigenvalues of finite multiplicity. Therefore, the sequence of discrete eigenvalues $\{\beta_k\}_{k=0}^M$ is strictly bounded from below by the essential spectrum and satisfies:
    \begin{equation*}
        -\min_{x \in \overline{\Omega}} \int_\Omega J(x-y) \, dy = \sup \sigma_c < \beta_M \leq \dots \leq \beta_1 < \beta_0 = 0.
    \end{equation*}
    If the sequence is infinite ($M = \infty$), the general spectral theory of self-adjoint operators dictates that the eigenvalues must accumulate exactly at the upper bound of the essential spectrum, meaning $\lim_{k \to \infty} \beta_k = \sup \sigma_c$. This concludes the proof of the Non-local Spectral Theorem.
\end{proof}

To the author's best knowledge, this is the first, rigorous result establishing existence of eigenfunctions corresponding to non-trivial eigenvalues on a bounded domains in the Neumann-type case.

\section{Spectral gap}

We begin by recalling the fundamental mass conservation property of the non-local Neumann operator.

\begin{rem}
    For any integral kernel $J$ satisfying Assumption \ref{asm:KernelProperties}, the dispersal operator $\mathcal{L}$ defined in \eqref{EQ:NonLocalNeumann} admits a trivial principal eigenvalue $\beta_0 = 0$. The corresponding eigenfunction is given by the constant function 
    \begin{equation}
        v_0(x) = \frac{1}{\sqrt{|\Omega|}}.
    \end{equation}
\end{rem}

Unlike the classical Laplacian, whose spectrum on a bounded domain is purely discrete, the non-local operator $\mathcal{L}$ lacks compactness. Consequently, its spectrum contains a continuous (essential) part.

\begin{thm}
    \label{thm:contninous_spectrum}
    Let $\Omega$ and $J$ satisfy Assumption 1.1. The essential (continuous) spectrum of the non-local operator $\mathcal{L}$, denoted by $\sigma_c := \sigma_c(\mathcal{L})$, is given by the range of its multiplication part
    \begin{equation}
        \sigma_c = \left\{ -\int_\Omega J(x-y) \,dy : x \in \overline{\Omega} \right\}.
    \end{equation}
\end{thm}
\begin{proof}
    For the detailed proof, we refer the reader to Theorem 2.1 in \cite{Tadej2025}.
\end{proof}
Each eigenvalue $\beta$ located strictly above this continuous band, i.e., $\beta > \sup \sigma_c$, belongs to the discrete spectrum and is isolated. We define the first non-trivial eigenvalue variationally as
\begin{equation}\label{eq:beta1_def}
    \beta_1 := \sup_{\|v\|_2 = 1 \\ \int_\Omega v(x) \:dx = 0} \langle \mathcal{L}v, v \rangle.
\end{equation}
The basic properties of $\beta_1$ are summarized in the following theorem.

\begin{thm}
    \label{thm:beta1_bounds}
    Let $\beta_1$ be defined in equation \eqref{eq:beta1_def}. Then the following bounds hold true
    \begin{equation}
        \sup \sigma_c \leq \beta_1 < 0.
    \end{equation}
\end{thm}
\begin{proof}
    For the detailed proofs of lower and upper bounds of $\beta_1$ we refer the reader to Proposition 3.4 and Lemma 3.5 in \cite{AndreuVaillo2010} respectively.
\end{proof}

While Theorem \ref{thm:beta1_bounds} guarantees that $\beta_1$ cannot be below the continuous spectrum, it does not prevent $\beta_1$ from collapsing into $\sigma_c$ (i.e., $\beta_1 = \sup \sigma_c$), which would result in the non-existence of the first non-trivial eigenfunction. To ensure the existence of an isolated eigenvalue $\beta_1$, we introduce a sufficient condition regarding the geometry of $\Omega$ and the kernel $J$.

\begin{prp}
    \label{prop:spectral_gap}
    Let $\Omega$ and $J$ satisfy Assumptions \ref{asm:KernelProperties} and \ref{asm:DomainProperties_1}.
    Then the following strict lower bound holds true
    \begin{equation*}
        \sup \sigma_c < \beta_1.
    \end{equation*}
\end{prp}
\begin{proof}
    We construct a test function $v \in A_1$ given by $v(x) = 1/\sqrt{|\Omega|}$ for $x \in \Omega_1$ and $v(x) = -1/\sqrt{|\Omega|}$ for $x \in \Omega_2$. Since $|\Omega_1|=|\Omega_2|$, we have $\int_\Omega v(x)\:dx = 0$ and $\|v\|_{L^2}=1$. Using the symmetrization identity, we compute the energy
    \begin{align*}
        \langle \mathcal{L}v, v \rangle &= -\frac{1}{2} \iint_{\Omega \times \Omega} J(x-y)(v(x)-v(y))^2 \, dx dy \\
        &= -2 \iint_{\Omega_1 \times \Omega_2} J(x-y) \left( \frac{2}{\sqrt{|\Omega|}} \right)^2 \, dx dy \\
        &= -\frac{4}{|\Omega|} \iint_{\Omega_1 \times \Omega_2} J(x-y) \, dx dy.
    \end{align*}
    On the other hand, $\sup \sigma_c = \max (-\int_\Omega J(x-y)dy) = - \min \int_\Omega J(x-y)dy$. 
    By the assumption \eqref{eq:sufficient_cond}, multiplying both sides by $-1$ yields $\langle \mathcal{L}v, v \rangle > \sup \sigma_c$. Taking the supremum over all $v \in L^2(\Omega)$ with $\|v\|_2 = 1$ and $\int_\Omega v(x)\:dx = 0$ concludes the proof.
\end{proof}

For the sake of exposition, we provide an example of the kernel $J$ and domain $\Omega$ satisfying Assumptions \ref{asm:KernelProperties} and \ref{asm:DomainProperties_1}.
\begin{example}
    \label{ex:generalized_exponential}
    Let $\Omega = (-L, L)$ and consider the family of symmetric, generalized exponential dispersal kernels given by
    \begin{equation}
        J_p(z) = C_p e^{-\lambda |z|^p}, \quad \lambda > 0, \quad p > 0,
    \end{equation}
    where $C_p$ is the normalization constant ensuring $\int_{\mathbb{R}} J_p(z) \, dz = 1$. The fulfillment of the sufficient condition for the spectral gap \eqref{eq:sufficient_cond} strongly depends on the shape parameter $p$:
    \begin{itemize}
        \item If $p \le 1$, the condition holds true for any domain length $L > 0$ and any dispersal parameter $\lambda > 0$.
        \item If $p > 1$, the condition holds true provided that the non-dimensional parameter $\eta = L\lambda^{1/p}$ is sufficiently large.
    \end{itemize}
\end{example}

\begin{proof}
    We split the domain into $\Omega_1 = (-L, 0)$ and $\Omega_2 = (0, L)$. Evaluating the right-hand side of Assumption \ref{asm:DomainProperties_1}, the minimum of the local potential $b(x) = \int_{-L}^L J_p(x-y) \, dy$ occurs at the boundaries $x = \pm L$, yielding
    \begin{equation*}
        \min_{x \in \overline{\Omega}} \int_{-L}^L J_p(x-y) \, dy = \int_{-L}^L J_p(L-y) \, dy = \int_0^{2L} J_p(z) \, dz.
    \end{equation*}
    For the cross-term on the left-hand side of \eqref{eq:sufficient_cond}, substituting $z = y-x$ and noticing that $|\Omega| = 2L$, we yield
    \begin{equation*}
        \frac{2}{L} \iint_{\Omega_1 \times \Omega_2} J_p(x-y) \, dy dx = \frac{2}{L} \int_0^L z J_p(z) \, dz + \frac{2}{L} \int_L^{2L} (2L-z) J_p(z) \, dz.
    \end{equation*}
    Let $\Delta(L, \lambda)$ denote the difference between the right-hand side and the left-hand side of \eqref{eq:sufficient_cond}. Splitting the integral over $(0, 2L)$ into two intervals $(0, L)$ and $(L, 2L)$, we rearrange the terms as follows:
    \begin{align}
        \begin{split}
            \label{eq:Delta_1st_set}
            \Delta(L, \lambda) &= \int_0^{2L} J_p(z) \, dz - \left( \frac{2}{L} \int_0^L z J_p(z) \, dz + \frac{2}{L} \int_L^{2L} (2L-z) J_p(z) \, dz \right) \\
            &= \int_0^L J_p(z) \, dz + \int_L^{2L} J_p(z) \, dz - \frac{2}{L} \int_0^L z J_p(z) \, dz - \frac{2}{L} \int_L^{2L} (2L-z) J_p(z) \, dz \\
            &= \int_0^L \left( 1 - \frac{2z}{L} \right) J_p(z) \, dz + \int_L^{2L} \left( 1 - \frac{2(2L-z)}{L} \right) J_p(z) \, dz. 
        \end{split} \\
    \intertext{Applying the substitution $u = 2L - z$ to the second integral and renaming the dummy variable $z$ to $u$ in the first integral, we obtain}
        \begin{split}
            \label{eq:Delta_2nd_set}
            \Delta(L, \lambda) &= \int_0^L \left( 1 - \frac{2u}{L} \right) J_p(u) \, du + \int_0^L \left( 1 - \frac{2u}{L} \right) J_p(2L-u) \, du \\
            &= \int_0^L \left( 1 - \frac{2u}{L} \right) \Big[ J_p(u) + J_p(2L-u) \Big] \, du \\
            &= \int_0^L \left( 1 - \frac{2u}{L} \right) F(u) \, du,
        \end{split} \\
    \intertext{where we denoted $F(u) := J_p(u) + J_p(2L-u)$. Using the translation $u = L/2 - t$, we shift the domain of integration to $[-L/2, L/2]$ and note that $du = -dt$}
        \begin{split}
            \label{eq:Delta_3rd_set}
            \Delta(L, \lambda) &= \int_{L/2}^{-L/2} \left( 1 - \frac{2(L/2 - t)}{L} \right) F(L/2 - t) \, (-dt) \\
            &= \int_{-L/2}^{L/2} \left( 1 - 1 + \frac{2t}{L} \right) F(L/2 - t) \, dt \\
            &= \int_{-L/2}^{L/2} \frac{2t}{L} F(L/2 - t) \, dt. 
        \end{split} \\
    \intertext{By splitting the integration domain into two symmetric halves and applying the change of variable $t \mapsto -t$ on the negative half, we obtain}
        \begin{split}
            \label{eq:Delta_4th_set}
            \Delta(L, \lambda) &= \int_{-L/2}^0 \frac{2t}{L} F(L/2 - t) \, dt + \int_0^{L/2} \frac{2t}{L} F(L/2 - t) \, dt \nonumber \\
            &= \int_{L/2}^0 \frac{-2t}{L} F(L/2 + t) \, (-dt) + \int_0^{L/2} \frac{2t}{L} F(L/2 - t) \, dt \nonumber \\
            &= -\int_0^{L/2} \frac{2t}{L} F(L/2 + t) \, dt + \int_0^{L/2} \frac{2t}{L} F(L/2 - t) \, dt \nonumber \\
            &= \int_0^{L/2} \frac{2t}{L} \Big[ F(L/2 - t) - F(L/2 + t) \Big] \, dt.
        \end{split}
    \end{align}
    Since $t > 0$ strictly inside the integration domain, the inequality $\Delta(L, \lambda) > 0$ holds if $F(z)$ is strictly decreasing on $(0, L)$. We compute the derivative
    \begin{equation*}
        F'(z) = J_p'(z) - J_p'(2L-z) = C_p \lambda p \left[ (2L-z)^{p-1} e^{-\lambda(2L-z)^p} - z^{p-1} e^{-\lambda z^p} \right].
    \end{equation*}
    To determine the sign of $F'(z)$, we define $h(z) = z^{p-1} e^{-\lambda z^p}$ and analyze its derivative for $z > 0$
    \begin{equation*}
        h'(z) = z^{p-2} e^{-\lambda z^p} \left[ (p-1) - \lambda p z^p \right].
    \end{equation*}

    \textbf{Case 1: $p \le 1$.} 
    Since $z > 0$, $\lambda > 0$, and $p-1 \le 0$, the bracket $\left[ (p-1) - \lambda p z^p \right]$ is strictly negative. Thus, $h'(z) < 0$ globally on $(0, \infty)$, meaning $h(z)$ is strictly decreasing. Because $z < 2L-z$ for all $z \in (0, L)$, we immediately obtain $h(2L-z) < h(z)$, which implies $F'(z) < 0$. Therefore we have shown that $\Delta(L, \lambda) > 0$ unconditionally for any $L > 0$ and $\lambda > 0$.

    \textbf{Case 2: $p > 1$.} 
    Here, $p-1 > 0$, meaning $h(z)$ initially increases before decaying. This causes $F'(z) > 0$ for small arguments. Consequently, Assumption \ref{asm:DomainProperties_1} fails for sufficiently small domains. 
    
    However, we can establish a rigorous global bound for the cross-term to find the sufficient conditions for the gap. Returning to the exact expression for the cross-term, we note that for the second integral in \eqref{eq:Delta_1st_set} where $z \in [L, 2L]$, the algebraic inequality $2L-z \le z$ holds. Applying this inequality yields
    \begin{equation}
        \label{eq:cross_term_bound_derivation}
        \frac{2}{L} \int_L^{2L} (2L-z) J_p(z) \, dz \le \frac{2}{L} \int_L^{2L} z J_p(z) \, dz.
    \end{equation}
    Summing the integrals over $(0, L)$ and $(L, 2L)$, the entire cross-term is bounded by the first moment of the right half of the kernel $m_1(\lambda) = \int_{-\infty}^\infty J_p(z) |z| \, dz$
    \begin{equation*}
        \frac{2}{L} \iint_{\Omega_1 \times \Omega_2} J_p(x-y) \, dy dx \le \frac{2}{L} \int_0^{2L} z J_p(z) \, dz < \frac{m_1(\lambda)}{L}.
    \end{equation*}
    
    To examine the influence of the parameter $\lambda$, which controls the variance of the kernel, we introduce a scaling argument. Utilizing the change of variables $u = \lambda^{1/p} z$, we can express the kernel as $J_p(z) = \lambda^{1/p} \tilde{J}_p(\lambda^{1/p} z)$, where $\tilde{J}_p(u) = \tilde{C}_p e^{-|u|^p}$ is the reference unscaled kernel. Consequently, the first moment scales as 
    \begin{equation*}
        m_1(\lambda) = \lambda^{-1/p} \int_{-\infty}^\infty u \tilde{J}_p(u) \, du = \lambda^{-1/p} \tilde{m}_1,
    \end{equation*}
    where $\tilde{m}_1$ is a finite constant dependent only on $p$. 
    
    Applying the same variable change to the local part (the right-hand side of \eqref{eq:sufficient_cond}), we obtain
    \begin{equation*}
        \int_0^{2L} J_p(z) \, dz = \int_0^{2 L \lambda^{1/p}} \tilde{J}_p(u) \, du.
    \end{equation*}
    We now define the non-dimensional parameter $\eta := L \lambda^{1/p}$. Using $\eta$, we can re-write the strict lower bound for our difference $\Delta(L, \lambda)$ as
    \begin{equation*}
        \Delta(L, \lambda) > \int_0^{2\eta} \tilde{J}_p(u) \, du - \frac{ \tilde{m}_1}{\eta}.
    \end{equation*}
    Since the reference kernel is normalized, the integral $\int_0^{2\eta} \tilde{J}_p(u) \, du$ converges strictly to $1/2$ as $\eta \to \infty$. Conversely, the penalty term $\frac{\tilde{m}_1}{\eta}$ decays to $0$. Thus, there exists a threshold $\eta_0 > 0$ such that for all $\eta > \eta_0$, the inequality $\Delta(L, \lambda) > 0$ is satisfied. 
\end{proof}

Now, we show that as a domain is uniformly scaled to infinity, the number $\sup \sigma_c$ remains bounded strictly away from zero.

\begin{prp}
    \label{prp:retained_mass_limit}
    Let $\Omega_0 \subset \mathbb{R}^n$ be any open, bounded, and non-empty domain satisfying the uniform interior sphere condition. For a scaling factor $\lambda > 0$, define the scaled domain $\Omega_\lambda = \{ \lambda x : x \in \Omega_0 \}$. Assume the dispersal kernel $J \in C^0(\mathbb{R}^n)$ satisfies Assumption \ref{asm:KernelProperties}.
    
    Then, the minimum retained mass of the kernel on $\Omega_\lambda$ is asymptotically bounded strictly away from zero
    \begin{equation}
        \liminf_{\lambda \to \infty} \left( \min_{x \in \overline{\Omega_\lambda}} \int_{\Omega_\lambda} J(x-y) \, dy \right) \geq \frac{1}{2} > 0.
    \end{equation}
\end{prp}

\begin{proof}
Let $x_\lambda \in \overline{\Omega_\lambda}$ be a point where the minimum retained mass is attained for a given scale $\lambda$
\begin{equation} \label{eq:min_mass_def}
    \min_{x \in \overline{\Omega_\lambda}} \int_{\Omega_\lambda} J(x-y) \, dy = \int_{\Omega_\lambda} J(x_\lambda - y) \, dy.
\end{equation}
By performing a change of variables $z = x_\lambda - y$, we shift the domain of integration. Let $D_\lambda = x_\lambda - \Omega_\lambda = \{ x_\lambda - y : y \in \Omega_\lambda \}$. The integral in \eqref{eq:min_mass_def} becomes:
\begin{equation}
    \int_{\Omega_\lambda} J(x_\lambda - y) \, dy = \int_{D_\lambda} J(z) \, dz.
\end{equation}

We analyze the geometric properties of the shifted domain $D_\lambda$ as $\lambda \to \infty$. By definition, the origin $0 \in \overline{D_\lambda}$ for all $\lambda > 0$, because $x_\lambda \in \overline{\Omega_\lambda}$. Furthermore, $D_\lambda$ is an open, bounded set derived from scaling $\Omega_0$ by $\lambda$ and shifting it. 

Consider an arbitrary sequence of scaling factors $\lambda_k \to \infty$ as $k \to \infty$, and let $x_{\lambda_k}$ be the corresponding sequence of minimum points. We can write $x_{\lambda_k} = \lambda_k \tilde{x}_k$, where $\tilde{x}_k \in \overline{\Omega_0}$. Since $\overline{\Omega_0}$ is compact, we can extract a subsequence (still denoted by $\lambda_k$) such that $\tilde{x}_k \to x^* \in \overline{\Omega_0}$.

We define the shifted reference domain $U_0 = x^* - \Omega_0$. Since $x^* \in \overline{\Omega_0}$, the origin $0 \in \overline{U_0}$. Because $\Omega_0$ is open and non-empty, and we assumed that the domain satisfies the uniform interior sphere condition, we can choose an open ball $B(p, r) \subset U_0$ such that the origin lies exactly on its boundary, meaning $|p| = r$. As $k \to \infty$, the scaled and shifted domain $D_{\lambda_k} = \lambda_k (\tilde{x}_k - \Omega_0)$ converges geometrically to the scaled domain $\lambda_k U_0$. Since $B(p, r) \subset U_0$, it follows that $D_{\lambda_k}$ asymptotically contains the expanding ball $B(\lambda_k p, \lambda_k r)$.

Consider an arbitrary point $z$ in the open half-space $H_p = \{ z \in \mathbb{R}^n : z \cdot p > 0 \}$ and analyze the condition $z \in B(\lambda_k p, \lambda_k r)$. This is equivalent to
\begin{equation} \label{eq:ball_condition}
    |z - \lambda_k p|^2 < (\lambda_k r)^2.
\end{equation}
Expanding the squared norm and dividing by $\lambda_k > 0$ yields
\begin{equation} \label{eq:ball_condition_expanded}
    \frac{|z|^2}{\lambda_k} - 2(z \cdot p) + \lambda_k |p|^2 < \lambda_k r^2 \iff \frac{|z|^2}{\lambda_k} - 2(z \cdot p) < \lambda_k (r^2 - |p|^2).
\end{equation}
Since we specifically chose the ball such that its boundary touches the origin ($|p| = r$), the right-hand side is exactly zero for all $k$: $\lambda_k (r^2 - |p|^2) = 0$. For the left-hand side, since $z \in H_p$, we have strictly $z \cdot p > 0$. As $k \to \infty$, the term $|z|^2 / \lambda_k \to 0$. Thus, the left-hand side converges to a strictly negative value
\begin{equation}
    \lim_{k \to \infty} \left( \frac{|z|^2}{\lambda_k} - 2(z \cdot p) \right) = -2(z \cdot p) < 0.
\end{equation}
Consequently, for any fixed $z \in H_p$, there exists an index $K \in \mathbb{N}$ such that inequality \eqref{eq:ball_condition_expanded} holds strictly for all $k \geq K$. This implies $z \in B(\lambda_k p, \lambda_k r) \subset D_{\lambda_k}$, and thus $\mathds{1}_{D_{\lambda_k}}(z) = 1$ for all $k \geq K$. Taking the limit infimum confirms pointwise convergence on $H_p$
\begin{equation}
    \liminf_{k \to \infty} \mathds{1}_{D_{\lambda_k}}(z) \geq \mathds{1}_{H_p}(z) \quad \text{for all } z \in \mathbb{R}^n.
\end{equation}
Since the integral kernel $J(z) \geq 0$ and $\int_{\mathbb{R}^n} J(z) \, dz = 1$, we can apply Fatou's Lemma to obtain
\begin{equation}
    \liminf_{k \to \infty} \int_{\mathbb{R}^n} J(z) \mathds{1}_{D_{\lambda_k}}(z) \, dz \geq \int_{\mathbb{R}^n} \liminf_{k \to \infty} \big( J(z) \mathds{1}_{D_{\lambda_k}}(z) \big) \, dz \geq \int_{H_p} J(z) \, dz.
\end{equation}
By Assumption \ref{asm:KernelProperties}, $J(z)$ is symmetric. Thus, integrating it over the half-space $H_p$ we get
\begin{equation}
    \int_{H_p} J(z) \, dz = \frac{1}{2}.
\end{equation}
Since this holds for any sequence $\lambda_k \to \infty$, we conclude that
\begin{equation}
    \liminf_{\lambda \to \infty} \left( \min_{x \in \overline{\Omega_\lambda}} \int_{\Omega_\lambda} J(x-y) \, dy \right) \geq \frac{1}{2}.
\end{equation}
\end{proof}

\begin{lem}
    \label{lem:large_domain_gap}
    Let the dispersal kernel $J \in C^0(\mathbb{R}^n)$ satisfy Assumption \ref{asm:KernelProperties}. Let $\Omega_0 \subset \mathbb{R}^n$ be any open, bounded domain. There exists a scaling factor $\lambda^* > 0$ such that for all $\lambda > \lambda^*$, the non-local operator $\mathcal{L}$ on the scaled domain $\Omega_\lambda = \{ \lambda x : x \in \Omega_0 \}$ admits a strictly positive spectral gap
    \begin{equation}
        \sup \sigma_c(\Omega_\lambda) < \sup_{v \in A_1(\Omega_\lambda)} \langle \mathcal{L}v, v \rangle =: \beta_1(\Omega_\lambda).
    \end{equation}
\end{lem}

\begin{proof}
    
    Let $\bar{x} = \frac{1}{|\Omega_0|} \int_{\Omega_0} y \, dy$ be the geometric barycenter of the reference domain $\Omega_0$. The barycenter of the scaled domain $\Omega_\lambda$ is exactly $\lambda \bar{x}$. We define a linear test function centered at the barycenter of the domain $\Omega$
    \begin{equation}
        v_\lambda(x) = c_\lambda (x_k - \lambda \bar{x}_k),
    \end{equation}
    where the coordinate axis $k \in \{1, \dots, n\}$ corresponds to the maximum moment of inertia of $\Omega_0$, defined as
    \begin{equation}
        I_k(\Omega_0) = \max_{1 \leq i \leq n} \int_{\Omega_0} (y_i - \bar{x}_i)^2 \, dy.
    \end{equation}
    
    The zero-mean condition required for $v_\lambda \in A_1(\Omega_\lambda)$ is satisfied by the definition of the barycenter
    \begin{equation}
        \int_{\Omega_\lambda} v_\lambda(x) \, dx = c_\lambda \int_{\Omega_\lambda} (x_k - \lambda \bar{x}_k) \, dx = 0.
    \end{equation}
    
    We determine the constant $c_\lambda$ by enforcing the $L^2$ normalization constraint $\|v_\lambda\|_{L^2}^2 = 1$. Applying the change of variables $x = \lambda y$
    \begin{equation}
        1 = \int_{\Omega_\lambda} c_\lambda^2 (x_k - \lambda \bar{x}_k)^2 \, dx = c_\lambda^2 \lambda^{n+2} \int_{\Omega_0} (y_k - \bar{x}_k)^2 \, dy = c_\lambda^2 \lambda^{n+2} I_k(\Omega_0).
    \end{equation}
    This yields $c_\lambda^2 = (\lambda^{n+2} I_k(\Omega_0))^{-1}$.
    
    We evaluate the quadratic form of the operator $\mathcal{L}$ on the test function using the symmetrization identity
    \begin{equation}
        \langle \mathcal{L}v_\lambda, v_\lambda \rangle = -\frac{1}{2} \iint_{\Omega_\lambda \times \Omega_\lambda} J(x-y) (v_\lambda(x) - v_\lambda(y))^2 \, dx dy.
    \end{equation}
    The difference evaluates to $v_\lambda(x) - v_\lambda(y) = c_\lambda(x_k - y_k)$. The barycenter terms completely cancel out. Substituting this difference and enlarging the integration domain for the inner integral from $\Omega_\lambda$ to the entire space $\mathbb{R}^n$, we establish a strict lower bound
    \begin{align} \label{eq:energy_scaling_bound}
        \langle \mathcal{L}v_\lambda, v_\lambda \rangle &\geq -\frac{c_\lambda^2}{2} \int_{\Omega_\lambda} \int_{\mathbb{R}^n} J(x-y) (x_k - y_k)^2 \, dy \, dx \nonumber \\
        &= -\frac{c_\lambda^2}{2} \int_{\Omega_\lambda} \int_{\mathbb{R}^n} J(z) z_k^2 \, dz \, dx \nonumber \\
        &\geq -\frac{c_\lambda^2}{2} |\Omega_\lambda| \int_{\mathbb{R}^n} J(z) |z|^2 \, dz \nonumber \\
        &= -\frac{c_\lambda^2}{2} \lambda^n |\Omega_0| m_2(J).
    \end{align}
    The finiteness of the second moment $m_2(J) = \int_{\mathbb{R}^n} J(z) |z|^2 \, dz < \infty$ is guaranteed by Assumption \ref{asm:KernelProperties}. Substituting the normalization constant $c_\lambda^2$ into \eqref{eq:energy_scaling_bound}, we obtain
    \begin{equation}
        \langle \mathcal{L}v_\lambda, v_\lambda \rangle \geq -\frac{\lambda^n |\Omega_0| m_2(J)}{2 \lambda^{n+2} I_k(\Omega_0)} = -\frac{|\Omega_0| m_2(J)}{2 I_k(\Omega_0)} \cdot \frac{1}{\lambda^2}.
    \end{equation}
    The term $\frac{|\Omega_0| m_2(J)}{2 I_k(\Omega_0)}$ is a strictly positive geometric constant entirely independent of the scaling factor $\lambda$. As $\lambda \to \infty$, the lower bound of the energy strictly approaches $0$ as $\mathcal{O}(\lambda^{-2})$. 
    
    From Proposition \ref{prp:retained_mass_limit} it follows that there exists $\delta > 0$ such that $\sup \sigma_c < - \delta$ for all $\lambda > \lambda_0$. Thus, there must exist a sufficiently large scaling factor $\lambda^* \geq \lambda_0$ such that for all $\lambda > \lambda^*$:
    \begin{equation}
        \langle \mathcal{L}v_\lambda, v_\lambda \rangle \geq -\frac{|\Omega_0| m_2(J)}{2 I_k(\Omega_0)} \cdot \frac{1}{\lambda^2} > -\delta \geq \sup \sigma_c(\Omega_\lambda).
    \end{equation}
    Since the supremum over the admissible space $A_1(\Omega_\lambda)$ must be at least as large as the energy of this specific test function, we conclude $\beta_1(\Omega_\lambda) \geq \langle \mathcal{L}v_\lambda, v_\lambda \rangle > \sup \sigma_c(\Omega_\lambda)$, establishing the existence of the strictly positive spectral gap.
\end{proof}

Having established the condition for the principal eigenvalue, we can generalize the scaling argument to guarantee the existence of higher-order spectral gaps.

\begin{lem}
    \label{lem:large_domain_gap_k}
    Let the dispersal kernel $J \in C^0(\mathbb{R}^n)$ satisfy Assumption \ref{asm:KernelProperties}. Let $\Omega_0 \subset \mathbb{R}^n$ be any open, bounded domain. For any given integer $k \geq 1$, assume that the first $k-1$ non-trivial eigenfunctions $v_1, \dots, v_{k-1}$ exist. There exists a scaling factor $\lambda_k^* > 0$ such that for all $\lambda > \lambda_k^*$, the non-local operator $\mathcal{L}$ on the scaled domain $\Omega_\lambda = \{ \lambda x : x \in \Omega_0 \}$ admits at least $k$ strictly positive spectral gaps:
    \begin{equation}
        \sup \sigma_c(\Omega_\lambda) < \beta_k(\Omega_\lambda) \leq \dots \leq \beta_1(\Omega_\lambda).
    \end{equation}
\end{lem}

\begin{proof}
    Let $\Omega_0$ be the reference domain. We construct a $(k+1)$-dimensional subspace of smooth functions on $\Omega_0$. Choose $k+1$ linearly independent functions $\phi_0, \phi_1, \dots, \phi_k \in C^1(\overline{\Omega_0})$, and let $W_0 = \text{span}\{\phi_0, \dots, \phi_k\}$. 
    Because $W_0$ is a finite-dimensional subspace composed of continuously differentiable functions on a bounded domain, all functions in $W_0$ are globally Lipschitz continuous. Furthermore, by the equivalence of norms in finite dimensions, there exists a uniform constant $L_k > 0$ such that for every function $u \in W_0$ satisfying the normalization $\|u\|_{L^2(\Omega_0)} = 1$, the Lipschitz bound holds:
    \begin{equation}
        (u(\tilde{x}) - u(\tilde{y}))^2 \leq L_k |\tilde{x} - \tilde{y}|^2 \quad \text{for all } \tilde{x}, \tilde{y} \in \Omega_0.
    \end{equation}

    For a given scaling factor $\lambda > 0$, we define the mapped $(k+1)$-dimensional subspace on the scaled domain $\Omega_\lambda$ as $W_\lambda = \{ v_\lambda(x) = \lambda^{-n/2} u(x/\lambda) : u \in W_0 \}$.
    The scaling factor $\lambda^{-n/2}$ ensures that the $L^2$ normalization is preserved. For any $v \in W_\lambda$ with $\|v\|_{L^2(\Omega_\lambda)} = 1$, the corresponding function $u \in W_0$ satisfies:
    \begin{equation}
        1 = \int_{\Omega_\lambda} (v(x))^2 \, dx = \int_{\Omega_\lambda} \lambda^{-n} (u(x/\lambda))^2 \, dx = \int_{\Omega_0} (u(\tilde{x}))^2 \, d\tilde{x} = \|u\|_{L^2(\Omega_0)}^2.
    \end{equation}

    We select a specific test function $v_\lambda \in W_\lambda$ that satisfies the constraints of the admissible set $A_k(\Omega_\lambda)$. By definition, $v \in A_k(\Omega_\lambda)$ requires $\|v\|_{L^2(\Omega_\lambda)} = 1$ and orthogonality to the first $k$ eigenfunctions $v_0, v_1, \dots, v_{k-1}$. 
    Imposing these $k$ orthogonality constraints on the $(k+1)$-dimensional subspace $W_\lambda$ yields a system of $k$ homogeneous linear equations for the $k+1$ expansion coefficients of the basis of $W_\lambda$. Since the number of unknowns strictly exceeds the number of equations, the system admits a non-trivial solution. Normalizing this solution provides a function $v_\lambda \in W_\lambda \cap A_k(\Omega_\lambda)$.
    
    Next, we evaluate the quadratic form of the operator $\mathcal{L}$ on the chosen test function $v_\lambda \in W_\lambda$ using the symmetrization identity
    \begin{equation}
        \langle \mathcal{L}v_\lambda, v_\lambda \rangle = -\frac{1}{2} \iint_{\Omega_\lambda \times \Omega_\lambda} J(x-y) (v_\lambda(x) - v_\lambda(y))^2 \, dx dy.
    \end{equation}
    Applying the definition of $v_\lambda$ and the uniform Lipschitz bound $L_k$ from the reference domain, the squared difference scales as
    \begin{align} \label{eq:lipschitz_scaling}
        (v_\lambda(x) - v_\lambda(y))^2 &= \lambda^{-n} \left( u\left(\frac{x}{\lambda}\right) - u\left(\frac{y}{\lambda}\right) \right)^2 \nonumber \\
        &\leq \lambda^{-n} L_k \left| \frac{x}{\lambda} - \frac{y}{\lambda} \right|^2 = \frac{L_k}{\lambda^{n+2}} |x-y|^2.
    \end{align}
    Substituting this bound into the energy functional and enlarging the integration domain for the inner integral from $\Omega_\lambda$ to the entire space $\mathbb{R}^n$, we establish a strict lower bound for the energy of any normalized function in $W_\lambda$
    \begin{align} \label{eq:energy_scaling_bound_k}
        \langle \mathcal{L}v_\lambda, v_\lambda \rangle &\geq -\frac{L_k}{2 \lambda^{n+2}} \int_{\Omega_\lambda} \int_{\Omega_\lambda} J(x-y) |x-y|^2 \, dy \, dx \nonumber \\
        &\geq -\frac{L_k}{2 \lambda^{n+2}} \int_{\Omega_\lambda} \int_{\mathbb{R}^n} J(x-y) |x-y|^2 \, dy \, dx \nonumber \\
        &= -\frac{L_k}{2 \lambda^{n+2}} \int_{\Omega_\lambda} \left( \int_{\mathbb{R}^n} J(z) |z|^2 \, dz \right) dx \nonumber \\
        &= -\frac{L_k}{2 \lambda^{n+2}} |\Omega_\lambda| m_2(J).
    \end{align}
    Since $|\Omega_\lambda| = \lambda^n$, the $\lambda^n$ terms cancel out, yielding a uniform lower bound for the entire subspace $W_\lambda$
    \begin{equation}
        \langle \mathcal{L}v_\lambda, v_\lambda \rangle \geq -\frac{L_k m_2(J)}{2 \lambda^2} =: - C(k,n,\Omega_0)  \frac{ m_2(J) }{\lambda^2}.
    \end{equation}

    From Proposition \ref{prp:retained_mass_limit}, it follows that there exists $\delta > 0$ such that $\sup \sigma_c(\Omega_\lambda) < - \delta$ for all $\lambda > \lambda_0$. Thus, there must exist a sufficiently large scaling factor $\lambda_k^* \geq \lambda_0$ such that for all $\lambda > \lambda_k^*$
    \begin{equation}
        \beta_k(\Omega_\lambda) \geq - C(k,n,\Omega_0) \frac{m_2(J)}{\lambda^2} > -\delta \geq \sup \sigma_c(\Omega_\lambda).
    \end{equation}
    This end the proof of Lemma.
\end{proof}

\section{Non-local eigenvalue problem}

We start, by constructing first non-trivial solution $(\beta_1, v_1) \in \mathbb{R} \times L^2(\Omega)$ of the non-local eigenvalue problem \eqref{eq:nonlocal_eigenvalue_problem}. The solution is obtained variationaly, that is we consider the optimization problem
\begin{equation}
    \label{eq:eigenvalue_problem_1}
    \beta_1 = \sup_{v \in A_1} \: \langle \mathcal L v, v \rangle,
\end{equation}
where the closed set of admissible functions $A_1 \subset L^2(\Omega)$ is defined by the formula
\begin{equation}
    \label{eq:A_set_Def}
    A_1 = \left\{ v \in L^2(\Omega) : \int_\Omega v(x) \:dx = 0, \: \int_\Omega v^2(x) \:dx = 1 \right\}
\end{equation}

\begin{lem}[Galerkin Approximation]
\label{lem:galerkin_approximation}
Let $V_N = \operatorname{span}\{\varphi_0, \dots, \varphi_N\} \subset L^2(\Omega)$ be a finite-dimensional subspace. The projection of the non-local eigenvalue problem \eqref{eq:nonlocal_eigenvalue_problem} onto $V_N$ is equivalent to the algebraic system 
\begin{equation}
    \label{eq:matrix_system_lemma}
    (\mathbf{A} - \mathbf{B})\mathbf{c} = \beta_1^N \mathbf{c},
\end{equation}
    where $\mathbf{A}, \mathbf{B} \in \mathbb{R}^{(N+1) \times (N+1)}$ are symmetric matrices and $\mathbf{c} \in \mathbb{R}^{N+1}$ is the vector of coefficients. Under the assumption $\varphi_0 = \frac{1}{\sqrt{|\Omega|}}$, the condition $v \in A_1 \cap V_N$ holds if and only if
    \begin{equation}
        \label{eq:system_constraint_lemma}
        c_0 = 0, \quad \mathbf{c}^T\mathbf{c} = 1.
    \end{equation}
\end{lem}

\begin{proof}

    Given an orthonormal set $\{\varphi_i\}_{i=0}^\infty$ in $L^2(\Omega)$, since $L^2(\Omega^2) = L^2(\Omega) \otimes L^2(\Omega)$, we expand function of two variables $K(x,y) = J(x-y)$ into series with symmetric coefficients $a_{k,l}$
    \begin{equation}
        \label{eq:kernel_expansion}
        K(x,y) = \sum_{k,l = 0}^\infty a_{k, l} \varphi_k(x) \varphi_l(y).
    \end{equation}
    Plugging \eqref{eq:kernel_expansion} into the equation \eqref{eq:nonlocal_eigenvalue_problem} we get
    \begin{equation}
        \label{eq:nonlocal-eigenvalue-problem-kernel-expansion}
        \sum_{k, l = 0}^\infty a_{k, l} \varphi_k(x)\int_\Omega \varphi_l(y) v_1(y)\:dy - \int_\Omega J(x-y)\:dy\: v_1(x) = \beta_1 v_1(x)
    \end{equation}
    For convenience, denote the function $b(x) = \int_\Omega J(x-y)\:dy$. We seek for a finite dimensional solution of the form
    \begin{equation}
        \label{eq:v_N_Def}
        v_1^N(x) = \sum_{m=0}^N c_m \varphi_m(x).
    \end{equation}
    Substituting formula \eqref{eq:v_N_Def} into equation \eqref{eq:nonlocal-eigenvalue-problem-kernel-expansion}, we obtain
    \begin{equation}
        \label{eq:nonlocal-eigenvalue-problem-v_N_plugged-kernel-expansion}
        \sum_{k, l = 0}^\infty \sum_{m=0}^N a_{k, l} c_m \varphi_k(x)\int_\Omega \varphi_l(y) \varphi_m(y)\:dy - \sum_{m=0}^N c_m b(x) \varphi_m(x) = \beta_1^N \sum_{m=0}^N c_m \varphi_m(x).
    \end{equation}
    For fixed $n \in \{0, \ldots, N\}$ we multiply equation \eqref{eq:nonlocal-eigenvalue-problem-v_N_plugged-kernel-expansion} by $\varphi_n$ and integrate in $\Omega$. Re-indexing and using orthogonality of $\{\varphi_i\}_{i=0}^N$, this yields the simplified form
    \begin{equation}
        \label{eq:coeffs_system}
        \sum_{l=0}^N(a_{n, l} - b_{l,n}) c_l = \beta_1^N c_n.
    \end{equation}
    For our convenience we denote the real, symmetric matrices $\mathbf{A}, \mathbf{B}$ of size $(N+1) \times (N+1)$ with $\mathbf{A}_{i,j} = a_{i,j}, \mathbf{B}_{i, j} = b_{i,j}$, and a solution vector $\mathbf{c} = (c_0, c_1, \ldots, c_N)^T$. Then, we re-write equation \eqref{eq:coeffs_system} as
    \begin{equation}
        \label{eq:matrix_system}
        (\mathbf{A} - \mathbf{B}) \mathbf{c} = \beta_1^N \mathbf{c}.
    \end{equation}
    Before we solve the eigenvalue problem \eqref{eq:matrix_system}, we need to impose the constraints as in the definition \eqref{eq:A_set_Def}. Namely, we get that the vector of coefficients $\mathbf{c}$ need to satisfy
    \begin{equation}
        \label{eq:matrix_system_constraint}
        c_0 = 0, \quad \mathbf{c}^T \mathbf{c} = 1.
    \end{equation}
    
\end{proof}

We provide the following Proposition regarding solvability of the problem \eqref{eq:matrix_system}-\eqref{eq:matrix_system_constraint}.
\begin{prp}
    \label{prp:finite_dimensional_minimizer}
    Let $\mathcal{C} = \left\{ \mathbf{c} \in \mathbb{R}^{N+1} : c_0 = 0, \, \|\mathbf{c}\| = 1 \right\}$ be the constraint set. 
    The variational problem
    \begin{equation}
        \label{eq:variational_discrete}
        \beta_1^N := \sup_{\mathbf{c} \in \mathcal{C}} \: \langle (\mathbf{A} - \mathbf{B}) \mathbf{c}, \mathbf{c} \rangle
    \end{equation}
    admits a maximizer $\mathbf{c}^* \in \mathcal{C}$. Moreover, the pair $(\beta_1^N, \mathbf{c}^*)$ is a solution to the algebraic system \eqref{eq:matrix_system}-\eqref{eq:matrix_system_constraint}.
\end{prp}

\begin{proof}
    Let $\mathbf{M} = \mathbf{A} - \mathbf{B}$. Since $\mathbf{A}$ and $\mathbf{B}$ are symmetric, $\mathbf{M}$ is a real symmetric matrix. The constraint $c_0=0$ implies that we are effectively looking for the spectrum of the submatrix $\tilde{\mathbf{M}} \in \mathbb{R}^{N \times N}$ with indices $1 \dots N$.
    
    Since the set $\mathcal{C}$ is compact and the map $\mathbf{c} \mapsto \langle \mathbf{M}\mathbf{c}, \mathbf{c}\rangle$ is continuous, the supremum in \eqref{eq:variational_discrete} is attained. Let $\beta_1^N$ denote this maximum value and $\mathbf{c}^*$ be a corresponding maximizer. By the min-max principle restricted to the subspace $\mathbb{R}^N \cong \{ \mathbf{c} \in \mathbb{R}^{N+1} : c_0=0 \}$, $\beta_1^N$ coincides with the largest eigenvalue of $\tilde{\mathbf{M}}$, and $\mathbf{c}^*$ is the associated eigenvector. Consequently, the pair satisfies $(\mathbf{A}-\mathbf{B})\mathbf{c}^* = \beta_1^N \mathbf{c}^*$ with the required constraints.
\end{proof}

Finally, we show that the sequence $(\beta_1^N, v_1^N)$ of solutions to the finite dimensional problem \eqref{eq:matrix_system}-\eqref{eq:matrix_system_constraint}, converges to the solution $(\beta_1, v_1)$ of the original problem \eqref{eq:nonlocal_eigenvalue_problem}. We begin with the convergence of eigenvalues.

\begin{lem}
    \label{lem:convergence_principal_eigenvalue}
    Let $\beta_1^N$ be the principal eigenvalue of the discrete problem \eqref{eq:variational_discrete}. Then
    \begin{equation*}
        \beta_1^N \rightarrow \beta_1 \quad \text{ as } N \rightarrow \infty.
    \end{equation*}
\end{lem}

\begin{proof}
    Let $V_N = \text{span}\{\varphi_0, \ldots, \varphi_N\}$ and $A_1$ be defined as in \eqref{eq:A_set_Def}. Since $V_N \cap A_1 \subset V_{N+1} \cap A_1 \subset A_1$, we have
    \begin{equation}
        \beta_1^N = \sup_{u \in V_N \cap A_1} \langle \mathcal{L} u, u \rangle \leq \sup_{u \in V_{N+1} \cap A_1} \langle \mathcal{L} u, u \rangle \leq \ldots \leq \beta_1.
    \end{equation}
    The sequence $\{\beta_1^N\}_{N \geq 1}$ is non-decreasing and bounded from above, thus it converges to a limit $\beta^* \leq \beta_1$. To prove that $\beta^* = \beta_1$, we utilize the density of the basis $\{\varphi_k\}$ in $L^2(\Omega)$. For any $u \in A_1$, there exists a sequence of projections $u_N \in V_N \cap A_1$ such that $u_N \to u$ in $L^2(\Omega)$. The continuity of the bilinear form associated with $\mathcal{L}$ yields
    \begin{equation}
        \lim_{N \to \infty} \beta_1^N = \lim_{N \to \infty} \sup_{u \in V_N \cap A_1} \langle \mathcal{L} u, u \rangle \geq \lim_{N \to \infty} \langle \mathcal{L} u_N, u_N \rangle = \langle \mathcal{L} u, u \rangle.
    \end{equation}
    Taking the supremum over $u \in A_1$ on the right-hand side confirms that $\lim_{N \to \infty} \beta_1^N = \beta_1$.
\end{proof}

Now, we prove the strong convergence of eigenfunctions.

\begin{lem}
    \label{lem:convergence_principal_eigenfunction}
    Let $v_1^N$ be the function defined by formula \eqref{eq:v_N_Def} with vector of coefficients $\mathbf{c}$ being the maximizer of the problem \eqref{eq:variational_discrete}. Then 
    \begin{equation*}
        v_1^N \rightarrow v_1 \quad \text{ as } N \rightarrow \infty.
    \end{equation*}
\end{lem}

\begin{proof}
    Recall that $\|v_1^N\|_{L^2} = 1$. By the Banach-Alaoglu theorem, there exists a subsequence, still denoted by $v_1^N$, and a function $v \in L^2(\Omega)$ such that $v_1^N \rightharpoonup v_1$ weakly in $L^2(\Omega)$. The zero mass constraint is preserved under weak convergence
    \begin{equation}
        \int_\Omega v_1(x) \:dx = \lim_{N \to \infty} \int_\Omega v_1^N(x) \:dx = 0.
    \end{equation}
    To identify the limit, consider an arbitrary test function $\psi \in V_M$. For $N \geq M$, the discrete equation reads $\langle \mathcal{L} v_1^N, \psi \rangle = \beta_1^N \langle v_1^N, \psi \rangle$. Passing to the limit $N \to \infty$, and using the self-adjointness of $\mathcal{L}$, we obtain
    \begin{align}
        \langle v_1, \mathcal{L} \psi \rangle &= \lim_{N \to \infty} \langle v_1^N, \mathcal{L} \psi \rangle = \lim_{N \to \infty} \langle \mathcal{L} v_1^N, \psi \rangle \\
        &= \lim_{N \to \infty} \beta_1^N \langle v_1^N, \psi \rangle = \beta_1 \langle v_1, \psi \rangle.
    \end{align}
    Since $\bigcup_{M} V_M$ is dense in $L^2(\Omega)$, this implies $\mathcal{L} v_1 = \beta_1 v_1$ in the weak sense. It remains to prove that the weak convergence $v_1^N \rightharpoonup v_1$ is strong in $L^2(\Omega)$. We decompose the non-local operator $\mathcal{L}$ into a compact integral part $\mathcal{K}$ and a local multiplication part $\mathcal{M}$:
    \begin{equation}
        \mathcal{L}u(x) = (\mathcal{K}u)(x) - (\mathcal{M}u)(x) = \int_\Omega J(x-y)u(y)\,dy - b(x)u(x),
    \end{equation}
    where $b(x) = \int_\Omega J(x-y)dy$. Theorem \ref{thm:contninous_spectrum} states that the essential spectrum is generated by the multiplication operator, yielding $\sup \sigma_c = \max_{x \in \overline{\Omega}} (-b(x))$. From the Galerkin approximation, we evaluate the energy sequence
    \begin{equation}
        \langle \mathcal{L} v_1^N, v_1^N \rangle = \beta_1^N \to \beta_1 \quad \text{as } N \to \infty.
    \end{equation}
    Using our decomposition of $\mathcal{L}$, we can rewrite this energy equation as:
    \begin{equation}
        \label{eq:weak_solution_sequence}
        \langle \mathcal{K} v_1^N, v_1^N \rangle - \int_\Omega b(x) (v_1^N(x))^2 \, dx = \beta_1^N.
    \end{equation}
    Rearranging the terms in equation \eqref{eq:weak_solution_sequence} to group the $L^2$ norms, we get
    \begin{equation} \label{eq:rearranged_energy}
        \int_\Omega (\beta_1^N + b(x)) (v_1^N(x))^2 \, dx = \langle \mathcal{K} v_1^N, v_1^N \rangle.
    \end{equation}
    
    Because $J \in C^0(\mathbb{R}^n)$ and $\Omega$ is bounded, the integral operator $\mathcal{K} : L^2(\Omega) \to L^2(\Omega)$ is a compact operator. A fundamental property of compact operators is that they map weakly convergent sequences into strongly convergent ones. Since $v_1^N \rightharpoonup v_1$ weakly in $L^2(\Omega)$, it follows that $\mathcal{K} v_1^N \to \mathcal{K} v_1$ strongly. This guarantees the convergence of the inner product
    \begin{equation}
        \label{eq:scalar_product_compact_operator}
        \lim_{N \to \infty} \langle \mathcal{K} v_1^N, v_1^N \rangle = \langle \mathcal{K} v_1, v_1 \rangle.
    \end{equation}
    Next, from Lemma \ref{lem:convergence_principal_eigenvalue} we have that $\lim_{N \to \infty} \beta_1^N = \beta_1$ thus we write
    \begin{align}
        \lim_{N \to \infty} \int_\Omega (\beta_1^N + b(x)) (v_1^N(x))^2 \, dx 
        &= \lim_{N \to \infty} \left( \beta_1^N \int_\Omega (v_1^N(x))^2 \, dx + \int_\Omega b(x) (v_1^N(x))^2 \, dx \right) \nonumber \\
        &= \lim_{N \to \infty} \left( \beta_1^N \cdot 1 + \int_\Omega b(x) (v_1^N(x))^2 \, dx \right) \nonumber \\
        &= \beta_1 + \lim_{N \to \infty} \int_\Omega b(x) (v_1^N(x))^2 \, dx \nonumber \\
        &= \lim_{N \to \infty} \left( \beta_1 \int_\Omega (v_1^N(x))^2 \, dx + \int_\Omega b(x) (v_1^N(x))^2 \, dx \right) \nonumber \\
        &= \lim_{N \to \infty} \int_\Omega (\beta_1 + b(x)) (v_1^N(x))^2 \, dx.
        \label{eq:weighted_L2_norm}
    \end{align}
    Applying formulas \eqref{eq:scalar_product_compact_operator} and \eqref{eq:weighted_L2_norm} into equation \eqref{eq:rearranged_energy} we obtain
    \begin{equation} 
        \label{eq:limit_compact}
        \int_\Omega (\beta_1 + b(x)) v_1^N(x)^2 \, dx = \langle \mathcal{K} v_1, v_1 \rangle.
    \end{equation}
    Now, recall that the weak limit $v_1$ satisfies the equation $\mathcal{L} v_1 = \beta_1 v_1$. By rewriting this exact relation using the operator decomposition and taking the inner product with $v_1$, we obtain:
    \begin{equation}
        \langle \mathcal{K} v_1, v_1 \rangle - \int_\Omega b(x) v_1(x)^2 \, dx = \beta_1 \int_\Omega v_1(x)^2 \, dx.
    \end{equation}
    Rearranging this yields
    \begin{equation} \label{eq:weak_limit_energy}
        \int_\Omega (\beta_1 + b(x)) v_1(x)^2 \, dx = \langle \mathcal{K} v_1, v_1 \rangle.
    \end{equation}
    Combining equations \eqref{eq:limit_compact} and \eqref{eq:weak_limit_energy}, we have rigorously proven that
    \begin{equation} \label{eq:norm_convergence}
        \lim_{N \to \infty} \int_\Omega (\beta_1 + b(x)) (v_1^N(x))^2 \, dx = \int_\Omega (\beta_1 + b(x)) v_1(x)^2 \, dx.
    \end{equation}
    
    Finally, we utilize the crucial spectral gap condition. From Proposition 2.1, we established that $\beta_1 > \sup \sigma_c = \max_{x \in \overline{\Omega}} (-b(x))$. This implies that for all $x \in \overline{\Omega}$:
    \begin{equation}
        \beta_1 + b(x) \geq \beta_1 - \sup \sigma_c > 0.
    \end{equation}
    Because the weight function $w(x) := \beta_1 + b(x)$ is strictly positive and bounded, it induces an equivalent norm on the Hilbert space $L^2(\Omega)$, defined by
    \begin{equation}
        \|u\|_{*} := \left( \int_\Omega (\beta_1 + b(x)) u(x)^2 \, dx \right)^{1/2}.
    \end{equation}
    
    Equation \eqref{eq:norm_convergence} demonstrates exactly that the integral of the squared sequence with this weight converges to the integral of the squared limit. 
    
    To rigorously prove that $v_1^N$ converges strongly to $v_1$ in the standard $L^2(\Omega)$ norm, we evaluate the limit of the weighted squared difference $v_1^N - v_1$. Because $\beta_1 + b(x) \geq c > 0$, we can bound the standard $L^2$ norm by this weighted integral. Expanding the square inside the integral yields:
    \begin{align}
        0 \leq c \|v_1^N - v_1\|_{L^2}^2 &= c \int_\Omega (v_1^N(x) - v_1(x))^2 \, dx \nonumber \\
        &\leq \int_\Omega (\beta_1 + b(x))(v_1^N(x) - v_1(x))^2 \, dx \nonumber \\
        &= \int_\Omega (\beta_1 + b(x))(v_1^N(x))^2 \, dx \nonumber \\
        &\quad - 2 \int_\Omega (\beta_1 + b(x))v_1^N(x)v_1(x) \, dx \nonumber \\
        &\quad + \int_\Omega (\beta_1 + b(x))(v_1(x))^2 \, dx.
    \end{align}
    
    We now take the limit as $N \to \infty$ for each of the three terms on the right-hand side. The limit of the first was already established in equation \eqref{eq:norm_convergence}. For the second term, notice that since the weight function is bounded  and $v_1^N \rightharpoonup v_1 \in L^2(\Omega)$, we have that
    \begin{equation}
        \lim_{N \to \infty} \int_\Omega v_1^N(x) \big[ (\beta_1 + b(x)) v_1(x) \big] \, dx = \int_\Omega v_1(x) \big[ (\beta_1 + b(x)) v_1(x) \big] \, dx.
    \end{equation}
    The third term is independent of $N$. Combining these limits yields that
    \begin{align}
        \lim_{N \to \infty} \int_\Omega (\beta_1 + b(x))(v_1^N(x) - v_1(x))^2 \, dx 
        &= \int_\Omega (\beta_1 + b(x))(v_1(x))^2 \, dx \nonumber \\
        &\quad - 2 \int_\Omega (\beta_1 + b(x))(v_1(x))^2 \, dx \nonumber \\
        &\quad + \int_\Omega (\beta_1 + b(x))(v_1(x))^2 \, dx \nonumber \\
        &= 0.
    \end{align}
    
    Since $0 \leq c \lim_{N \to \infty} \|v_1^N - v_1\|_{L^2}^2 \leq 0$ and $c > 0$, it immediately follows by the squeeze theorem that:
    \begin{equation}
        \lim_{N \to \infty} \|v_1^N - v_1\|_{L^2}^2 = 0.
    \end{equation}
    Thus, $v_1^N \to v_1$ strongly in $L^2(\Omega)$. Because strong convergence preserves the standard $L^2$-norm, we finally conclude $\|v_1\|_{L^2} = \lim_{N \to \infty} \|v_1^N\|_{L^2} = 1$, confirming that $v_1 \in A_1$ is a valid, non-trivial normalized eigenfunction.
\end{proof}

Naturally, we are now interested in existence of subsequent eigenpairs $(\beta_k, v_k) \in \mathbb{R}\times A_k$ solving equation \eqref{eq:nonlocal_eigenvalue_problem}, with the $k$-th set of admissible functions $A_k$ defined by relation
\begin{align}
    \begin{split}
        A_k &= \left\{ v \in L^2(\Omega): v \perp \text{span}\{v_0, \ldots, v_{k-1}\}, \: \int_\Omega v^2(x) \: dx = 1 \right\} \quad k > 1.
    \end{split}
\end{align}
We will achieve this goal using the induction. The following Lemma serves as an induction step.

\begin{lem}
    \label{lem:inductive_step}
    Suppose that there exist $k$ pairs of eigenvalues and normalized orthonormal eigenfunctions $(\beta_j, v_j) \in \mathbb{R} \times A_j$ for $j=1, \ldots, k$ solving problem (3.1). Assume that
    \begin{equation}
        \label{eq:next_beta_inequality}
        \sup \sigma_c < \sup_{v \in A_{k+1}} \langle \mathcal{L}v, v \rangle =: \beta_{k+1}.
    \end{equation}
    Then, there exists an eigenfunction $v_{k+1} \in A_{k+1}$ corresponding to the eigenvalue $\beta_{k+1}$ such that
    \begin{equation}
        \label{eq:next_eigenproblem_solution}
        \mathcal{L}v_{k+1} = \beta_{k+1} v_{k+1}.
    \end{equation}
\end{lem}

\begin{proof}
    Let $k \geq 1$. We proceed using the approximation technique analogous to Lemmas \ref{lem:galerkin_approximation}, \ref{lem:convergence_principal_eigenvalue} and \ref{lem:convergence_principal_eigenfunction}. Let us consider the truncated problem on the subspace $A_{k+1} \cap V_N$ with $V_N = \text{span}\{\varphi_0, \ldots, \varphi_N\}$ and $N \geq k$. The discrete counterpart of the operator $\mathcal{L}$ again corresponds to the symmetric matrix $\mathbf{M} = \mathbf{A} - \mathbf{B}$ of size $(N+1) \times (N+1)$ restricted to the subspace $\mathcal{C} = \{ \mathbf{c} \in \mathbb{R}^{N+1} : c_0 = \ldots = c_k = 0, \|\mathbf{c}\| = 1\}$. We know that $\mathbf{M}$ has a complete orthonormal set, composed of eigenvectors $\mathbf{c}_1, \ldots, \mathbf{c}_N$ corresponding to eigenvalues $\beta_1^N \geq \beta_2^N \geq \ldots \geq \beta_N^N$. 

    Firstly, the sequence of eigenvalues $\{\beta_{k+1}^N\}_{N \geq k+1}$ is non-decreasing and bounded by $\beta_{k+1}$. Moreover, for every $u \in A_{k+1}$ there exists a sequence $u_i \rightarrow u$ converging strongly in $L^2(\Omega)$ as $i \rightarrow \infty$. Testing the expression \ref{eq:next_beta_inequality} with $u_i$, we get that $\beta_{k+1}^N \geq \langle \mathcal{L} u_i, u_i \rangle$. Now passing to the limit and taking supremum, we get that
    $\beta_{k+1}^N \rightarrow \beta_{k+1}$ as $N \rightarrow \infty$.
    
    Denote the function constructed from the vector $\mathbf{c}_{k+1} = (c_0^{k+1}, c_1^{k+1}, \ldots, c_N^{k+1})$ by the formula
    \begin{equation}
        v_{k+1}^N(x) = \sum_{m=0}^N c^{k+1}_m \varphi_m(x).
    \end{equation}
    Note that $\{v_{k+1}^N\}_{N \geq k+1}$ is bounded in $L^2(\Omega)$, thus up to a subsequence, it converges weakly to a function $v_{k+1} \in L^2(\Omega)$. Our goal is to show, that this seqence converges strongly, and $v_{k+1} \in A_{k+1}$.

    For any $j \in \{0, 1, \ldots, k\}$, utilizing the strong convergence $v_j^N \to v_j$ and the discrete orthogonality $\langle v_{k+1}^N, v_j^N \rangle = 0$ in the finite-dimensional subspace $V_N$, we can strictly bound the inner product
    \begin{align}
        |\langle v_{k+1} , v_j \rangle| &= \lim_{N \rightarrow \infty} |\langle v_{k+1}^N, v_j \rangle| \nonumber \\
        &= \lim_{N \rightarrow \infty} |\langle v_{k+1}^N, v_j^N \rangle + \langle v_{k+1}^N, v_j - v_j^N \rangle| \nonumber \\
        &\leq \lim_{N \rightarrow \infty} \left( 0 + \|v_{k+1}^N\|_2 \|v_j - v_j^N\|_2 \right) \\
        &= \lim_{N \rightarrow \infty} \|v_j - v_j^N\|_2 = 0. \nonumber
    \end{align}
    
    Now, we verify that $v_{k+1}$ is a weak solution of the eigenvalue problem. Fix $\psi \in A_{k+1}$, using Lemma \ref{lem:convergence_principal_eigenvalue} and weak convergence of $v_{k+1}^N$, we get
    \begin{align}
        \begin{split}
            \langle v_{k+1}, \mathcal{L} \psi \rangle
            &= \lim_{N \to \infty} \langle v_{k+1}^N, \mathcal{L} \psi \rangle = \lim_{N \to \infty} \langle \mathcal{L} v_{k+1}^N, \psi \rangle \\
            &= \lim_{N \to \infty} \beta_{k+1}^N \langle v_{k+1}^N, \psi \rangle = \beta_{k+1} \langle v_{k+1}, \psi \rangle.
        \end{split}
    \end{align}
    
    Now, to prove strong convergence of $v_{k+1}^N$ to $v_{k+1}$, we use the same technique, as in the proof of strong convergence in Lemma \ref{lem:convergence_principal_eigenfunction}. In particular, using compactness of the integral part of $\mathcal{L}$ we obtain
    \begin{equation}
        \label{eq:next_eigenfunction_convergence_weighted}
        \lim_{N \rightarrow \infty} \int_\Omega (\beta_{k+1} + b(x)) (v_{k+1}^N)^2\:dx = \int_\Omega (\beta_{k+1} + b(x))v_{k+1}^2\:dx.
    \end{equation}
    As we assumed that $\sup \sigma_c < \beta_{k+1}$, there exists a constant $c > 0$, such that 
    \begin{equation}
        0 < c < \beta_{k+1} - \sup \sigma_c \leq \beta_{k+1} + b(x)
    \end{equation}
    Hence, using the weak convergence of $v_{k+1}^N$ and formula \eqref{eq:next_eigenfunction_convergence_weighted}, we get that
    \begin{align}
        \begin{split}
            0 \leq c \lim_{N\rightarrow \infty}\|v_{k+1}^N - v_{k+1}\|_{L^2}^2 
            &= \lim_{N\rightarrow \infty} c \int_\Omega (v_{k+1}^N(x) - v_{k+1}(x))^2 \, dx \nonumber \\
            &\leq \lim_{N\rightarrow \infty} \int_\Omega (\beta_{k+1} + b(x))(v_{k+1}^N(x) - v_{k+1}(x))^2 \, dx \nonumber \\
            &= \lim_{N\rightarrow \infty} \int_\Omega (\beta_{k+1} + b(x))(v_{k+1}^N(x))^2 \, dx \nonumber \\
            &- 2 \lim_{N\rightarrow \infty} \int_\Omega (\beta_1 + b(x))v_{k+1}^N(x)v_{k+1}(x) \, dx \nonumber \\
            &+ \int_\Omega (\beta_1 + b(x))(v_{k+1}(x))^2 \, dx \\
            &= 0.
        \end{split}
    \end{align}
    This proves that $v_{k+1}^N \rightarrow v_{k+1}$ strongly in $L^2(\Omega)$ as $N \rightarrow \infty$. This concludes the Lemma.
\end{proof}

\bibliographystyle{siam}
\bibliography{bibliography}

\end{document}